\documentclass[12pt, leqno,fleqn]{amsart}
\usepackage{amsmath}
\usepackage{amssymb}
\usepackage{amsfonts}
\usepackage{graphicx}
\usepackage{amsthm}
\usepackage{enumerate}
\usepackage{lscape}
\usepackage{dsfont}
\usepackage{color}
\usepackage{mathtools}

\usepackage{setspace}
\newcommand{\CP}{\mathds{C}\mathrm{P}}

\newcommand{\CH}{\mathds{C}\mathrm{H}}

\newcommand{\C}{\mathds{C}}

\newcommand{\K}{K\"{a}hler}

\newtheorem{theor}{Theorem}
\newtheorem*{criterion}{Calabi's criterion}

\newtheorem{lem}[theor]{Lemma}

\newtheorem{ex}{Example}
\newtheorem{remark}[theor]{Remark}

\begin{document}

\title[On Calabi's diastasis function of the cigar metric]{On Calabi's diastasis function of the cigar metric}

\author{Andrea Loi}
\address{(Andrea Loi) Dipartimento di Matematica e Informatica \\
         Universit\`a di Cagliari (Italy)}
         \email{loi@unica.it}

\author{Michela Zedda}
\address{(Michela Zedda) Dipartimento di Matematica e Fisica ``Ennio De Giorgi'' \\
          Universit\`a del Salento (Italy)}
\email{michela.zedda@gmail.com}

\thanks{
The  authors were supported  by INdAM--GNSAGA - Gruppo Nazionale per le Strutture Algebriche, Geometriche e le loro Applicazioni.}

\begin{abstract}
We show that the  Cigar metric on $\C$ is an example of real analytic  K\"ahler manifold with globally defined and positive Calabi's diastasis function which  cannot be K\"ahler immersed into any 
(finite or infinite dimensional) complex space form.
\end{abstract}
 
\maketitle


\section{Introduction}
In his seminal paper \cite{calabi}, E. Calabi provides a criterion for a K\"ahler manifold $(M,g)$ to admit a holomorphic and isometric (from now on K\"ahler) immersion into a complex space form, finite or infinite dimensional. Recall that a complex space form, that we assume to be complete and simply connected, up to homotheties can be of three types, according to the sign of the constant holomorphic sectional curvature:
\begin{enumerate}
\item[($i$)] the complex Euclidean space $\C^N$ of complex dimension $N\leq \infty$, endowed with the flat metric.  Here $\C^\infty$ denotes the Hilbert space $\ell ^2({\C})$ consisting of sequences $w_j$, $j=1,2,\dots$, $w_j\in\C$ such that $\sum_{j=1}^{+\infty}|w_j|^2<+\infty$. 
\item[($ii$)] the complex projective space $\CP^N$ of complex dimension $N\leq \infty$, with the Fubini-Study metric $g_{FS}$.
\item[($iii$)] the complex hyperbolic space $\CH^N$ of complex dimension $N\leq \infty$, that is the unit ball $B\subset\C^N$ given by $B=\left\{(z_1,\dots,z_N)\in\C^N,\; \sum_{j=1}^N|z_j|^2<1\right\}$
 endowed with the hyperbolic metric.
\end{enumerate}


The criterion Calabi states is given in terms of the {\em diastasis function} associated to the metric $g$, defined as follows. By duplicating the variables $z$ and $\bar z$, a K\"ahler potential $\Phi$ around $p\in M$ can be complex analytically continued to a function $\tilde \Phi$ defined in a neighborhood $U$ of the diagonal containing $(p,\bar p)\in M\times M$.
The diastasis function is defined by the formula:
\begin{equation}\label{diast}
D^g(p,q)=\tilde\Phi(p,\bar p)+\tilde\Phi(q,\bar q)-\tilde\Phi(p,\bar q)-\tilde\Phi(q,\bar p),\quad (p,q)\in U.
\end{equation}
Once $p$ is fixed and $z$ are coordinates around it, $D^g(p,q)=D_p^g(z)$ is a K\"ahler potential in a neighborhood of $p$. The existence of a local K\"ahler immersion into a complex space form depends only on the derivatives with respect to $z$ and $\bar z$ of $D^g_p(z)$, evaluated at $p$. In particular, when $(M,g)$ is a complex curve endowed with a radial metric $g$ and the ambient space is the complex projective space, the criterion reads as follows:
\begin{criterion}
Let $(M,g)$ be a real analytic K\"ahler curve such that in a coordinate system $\{z\}$ centered at $p\in M$, $D_p^g(z)$ depends only on $|z|^2$. Then 
$(M,g)$ admits a local K\"ahler immersion into $\CP^{N\leq \infty}$ iff for any $n>0$:
$$
\frac{\partial^{2n}\exp\left(D_0(|z|^2)\right)}{\partial z^n\partial \bar z^n}|_0\geq 0.
$$
Further, the image of $M$ is not contained in any totally geodesic submanifold of the $N$-dimensional ambient space iff the number of derivatives in the above inequality different from zero are exactly $N$.
\end{criterion}
A local K\"ahler immersion can be extended to a global one iff for each point $p\in M$, the maximal analytical extension of $D_p^g(z)$ is single valued, condition which is fulfilled when $M$ is simply connected.

  In this paper, we are interested in studying how the multiplication of the K\"ahler metric $g$ by a positive constant $c$ affects the existence of a (local) K\"ahler immersion of $(M,cg)$ into complex space forms.
  
 Observe that when one studies K\"ahler immersions into the complex Euclidean space, the multiplication of the metric $g$ by $c$ is harmless. In fact, if $f:M\rightarrow \C^N$, $N\leq\infty$, satisfies $f^*(g_0)=g$ then $(\sqrt{c}f)^*(g_0)=cg$. 
  
A very interesting example in this sense is given by Cartan domains $\Omega$, i.e. irreducible bounded homogeneous domains of $\mathds{C}^d$ endowed with their Bergman metric $g_B$. The existence of a K\"ahler immersion of $(\Omega, c\,g_B)$ into $\CP^\infty$  depends firmly on $c$ and it is strictly related to the Wallach set $W(\Omega)$ of $\Omega$ (see \cite{articwall}). More precisely, if $\gamma$ is the genus of $\Omega$, then $(\Omega,cg_B)$, $c>0$ admits a K\"ahler immersion into $\CP^\infty$ if and only if $c\gamma$ belongs to $W(\Omega)\setminus\{0\}$. Recall that $W(\Omega)$ is a subset of the real line depending on two of the domain's invariants, denoted  by $a$ (strictly positive real number) and  $r$ (the rank of $\Omega$), and it is given by:
$$
W(\Omega)=\left\{0,\,\frac{a}{2},\,2\frac{a}{2},\,\dots,\,(r-1)\frac{a}{2}\right\}\cup \left((r-1)\frac{a}{2},\,\infty\right).
$$
(the author is referred to \cite{arazy}, \cite{faraut} and \cite{upmeier} for more details and results about the Wallach set).
This result has been generalized by the first author and R. Mossa \cite{loimossa12} for bounded homogeneous domains.

In view of the case of Hermitian symmetric spaces it is natural and interesting to exhibit examples of manifolds $(M,c g)$ that do not admit a K\"ahler immersion into any complex space form for any value of $c>0$. 
At the end of his paper  Calabi  himself provides the following two examples (in the first one $M$ is compact and in the second one $M$ is noncompact and complete).

\begin{ex}\rm
Consider the product $\CP^1\times\CP^1$ endowed with the metric $g=b_1 g_{FS}\oplus b_2 g_{FS}$, with $b_1$, $b_2$ positive real numbers such that $b_2/b_1$ is irrational. Then $(\CP^1\times\CP^1,cg)$ does not admit a K\"ahler immersion into $\CP^\infty$ for any value of $c$.
In fact, in \cite[Th.13]{calabi}, Calabi proves that $(\CP^n,cg_{FS})$ admits a K\"ahler immersion into $\CP^\infty$ iff $1/c$ is a positive integer, and this property cannot be fulfilled by both $1/cb_1$ and $1/cb_2$.
\end{ex} 

\begin{ex}\rm
Consider on $\mathds{C}$ the metric $g$ whose associate K\"ahler form $\omega$ is given by:
$\omega=\left(4\cos(z-\bar z)-1\right)dz\wedge d\bar z$. A K\"ahler immersion of $(\mathds{C},cg)$ into $\CP^\infty$ is not possible since the diastasis:
$$
D(p,q)=4\left[\cos(p-\bar p)+\cos(q-\bar q)-\cos(p-\bar q)-\cos(q-\bar p)\right]-|p-q|^2,
$$
takes negative values, e.g. for $q=p+2\pi$.
\end{ex}

These two examples suggest a refinement of the previous problem. In fact, both the metrics described present geometrical obstructions to the existence of a K\"ahler immersion into $\CP^\infty$ that put aside the role of $c$. More precisely, in the first example, the  \K\ form  $\omega$ associated to $g$ is not integral, while in the second one the diastasis associated to $g$ is negative at some points. Thus, it is interesting  to find examples of real analytic K\"ahler manifold $(M, cg)$ which cannot be  \K\ immersed into any (finite or infinite dimensional) complex space form for any $c>0$ and satisfy:
\begin{itemize}
\item[($i$)]
the \K\ form  $\omega$ associated to $g$ is integral;
\item[($ii$)]
the diastasis associated to $g$ is globally defined on $M\times M$ and positive.
\end{itemize}

Our main result in this direction  is the following theorem, proved at the end of next section:

\begin{theor}\label{mainnew}
Let $g=\frac{1}{1+|z|^2}dz\otimes d\bar z$ be the Cigar metric on $\C$. Then the diastasis function of the metric  $g$ is globally defined and positive  on $\C\times \C$ and  $(\C, c g)$ cannot be (locally) \K\ immersed   into any complex space form for any $c>0$.
\end{theor}

\begin{remark}\rm
It is worth pointing out that the cigar metric has positive sectional curvature. Hence, in view of Theorem \ref{mainnew},  it is  interesting to see if there exist examples  
of  negatively curved real analytic \K\ manifolds $(M ,g)$  with  globally  defined diastasis function which is positive and  such that 
$(M, cg)$ cannot be locally \K\ immersed into any complex space form for all $c>0$. 
\end{remark}

%
%
%

\section{The Cigar metric on $\mathds{C}$ and the proof of Theorem \ref{mainnew}}
The Cigar metric $g$ on $\mathds{C}$ has been introduced by Hamilton in \cite{Hamilton} as first example of K\"ahler--Ricci soliton on non-compact manifolds. It is defined by:
$$
g=\frac{dz\otimes d\bar z}{1+|z|^2}.
$$
In \cite{cigar} the authors of the present paper study K\"ahler--Ricci solitons, with the Cigar metric as particular case, from the symplectic point of view. 
A (globally defined) K\"ahler potential for this metric is given by (see also \cite{Suzuki}):
$$
D_0(|z|^2)=\int_{0}^{|z|^2}\frac{\log(1+s)}{s}ds,
$$
whose power series expansion around the origin reads:
\begin{equation}\label{expdiast}
D_0(|z|^2)= \sum_{j=1}^\infty (-1)^{j+1}\frac{|z|^{2j}}{j^2}.
\end{equation}
By duplicating the variable in this last expression, by \eqref{diast} we get:
\begin{equation}\label{diastdef}
D^g(z,w)=\sum_{j=1}^\infty\frac{(-1)^{j+1}}{j^2}\left(|z|^{2j}+|w|^{2j}-(z\bar w)^{2j}-(w\bar z)^{2j}\right).
\end{equation}
In the following lemma we prove that $D^g(z,w)$ is everywhere nonnegative and globally defined on $\mathds{C}\times\C$. 
It is worth pointing out that the fact that $D^g(z,w)$ is globally defined was already observed  in \cite{Suzuki}.
\begin{lem}\label{pos}
The diastasis function \eqref{diastdef} of the Cigar metric is globally defined and nonnegative.
\end{lem}
\begin{proof}
If we denote by ${\rm Li}_n(z)$ the polylogarithm function, defined for $|z|<1$ by:
$$
{\rm Li}_n(z)=\sum_{j=1}^\infty\frac{z^j}{j^n},
$$
and by analytic continuation otherwise, from \eqref{diastdef} we can write $D^g(z,w)$ as:
$$
D(z,w)=-{\rm Li}_2(-|z|^2)-{\rm Li}_2(-|w|^2)+{\rm Li}_2(-z\bar w)+{\rm Li}_2(-w\bar z).
$$
Write $z=\rho_1e^{i\theta_1}$ and $w=\rho_2e^{\theta_2}$ and let $\alpha=\theta_1-\theta_2$. Then:
\begin{equation}
\begin{split}\label{diastcigarli}
D(z,w)=&-{\rm Li}_2(-\rho_1^2)-{\rm Li}_2(-\rho_2^2)+{\rm Li}_2(-\rho_1\rho_2\,e^{i\alpha})+{\rm Li}_2(-\rho_1\rho_2\,e^{-i\alpha})\\
=&-{\rm Li}_2(-\rho_1^2)-{\rm Li}_2(-\rho_2^2)+2{\rm ReLi}_2(-\rho_1\rho_2\,e^{i\alpha}),
\end{split}
\end{equation}
where we are allowed to take the real parts since $D^g(z,w)$ is real. 
In order to simplify the term ${\rm ReLi}_2(-\rho_1\rho_2\,e^{i\alpha})$, we recall the following formula due to Kummer (see \cite{kummer} or \cite[p.15]{lewin}):
\begin{equation}
\begin{split}
{\rm ReLi}_2(\rho e^{i\theta})=&\frac12\left({\rm Li}_2(\rho e^{i\theta})+\overline{{\rm Li}_2(\rho e^{i\theta})}\right)\\
=&-\frac12\left(\int_0^\rho\frac{\log(1-y e^{i\theta})}{y}dy+\int_0^\rho\frac{\log(1-y e^{-i\theta})}{y}dy\right)\\
=&-\frac12\int_0^{\rho} \frac{\log(1-2y\cos\theta+y^2)}{y}dy
\end{split}\nonumber
\end{equation}
i.e.:
$$
{\rm ReLi}_2(-\rho e^{i\alpha})=-\frac12\int_0^{\rho} \frac{\log(1+2y\cos(\alpha)+y^2)}{y}dy.
$$ 
Since $1+2y\cos(\alpha)+y^2$ is decreasing for $0<\alpha<\pi$ and increasing for $\pi<\alpha<2\pi$, $\alpha=\pi$ is a minimum. Thus:
\begin{equation}\label{reli}
{\rm ReLi}_2(-\rho e^{i\alpha})\geq-\int_0^{\rho} \frac{\log(|1-y|)}{y}dy,
\end{equation}
where:
$$
-\int_0^{\rho} \frac{\log(|1-y|)}{y}dy=\begin{cases}{\rm Li}_2(\rho)& \textrm{if}\ \rho\leq 1\\
\frac{\pi^2}{6}-{\rm Li}_2(1-\rho)-\ln(\rho-1)\ln(\rho)& \textrm{otherwise}.
\end{cases}
$$
Thus, when $\rho_1\rho_2\leq1$ from \eqref{diastcigarli} and \eqref{reli}, we get:
$$
D(z,w)\geq-{\rm Li}_2(-\rho_1^2)-{\rm Li}_2(-\rho_2^2)+2{\rm Li}_2(\rho_1\rho_2)\geq 0,
$$
where the last equality follows since all the factors in the sum are positive. When $\rho_1\rho_2>1$, from \eqref{diastcigarli} and \eqref{reli}, we get:
\begin{equation}\label{lastd}
D(z,w)\geq-{\rm Li}_2(-\rho_1^2)-{\rm Li}_2(-\rho_2^2)+\frac{\pi^2}{3}-2{\rm Li}_2(1-\rho_1\rho_2)-2\ln(\rho_1\rho_2-1)\ln(\rho_1\rho_2).
\end{equation}
The RHS is positive for $1<\rho_1\rho_2\leq2$ since it is sum of positive factors. When $\rho_1\rho_2>2$, since all the factors are monotonic, it is enough to consider the limit as $\rho_1$ goes to $+\infty$ of $-D(z,w)/{\rm Li}_2(-\rho_1^2)$. By \eqref{lastd} above we get:
$$
\lim_{\rho_1\rightarrow+\infty}\frac{D(z,w)}{-{\rm Li}_2(-\rho_1^2)}\geq\frac52,
$$
and we are done.
\end{proof}

In order to prove Theorem \ref{mainnew} we need the following definition and properties of Bell polynomials. The partial Bell polynomials $B_{n,k}(x):=B_{n,k}(x_1,\dots, x_{n-k+1})$ of degree $n$ and weight $k$ are defined by (see e.g. \cite[p. 133]{comtet}):
\begin{equation}\label{bjkdef}
B_{n,k}(x_1,\dots, x_{n-k+1})=\sum_{\pi(k)}\frac{n!}{s_1!\dots s_{n-k+1}!}\left(\frac{x_1}{1!}\right)^{s_1}\left(\frac{x_2}{2!}\right)^{s_2}\cdots \left(\frac{x_{n-k+1}}{(n-k+1)!}\right)^{s_{n-k+1}},
\end{equation}
where the sum is taken over the integers solutions of:
$$
\begin{cases}s_1+2s_2+\dots+ks_{n-k+1}=n\\ s_1+\dots+s_{n-k+1}=k.\end{cases}
$$
Bell polynomials satisfy the following equalities, which are fundamental for the proof of our result (the second one has been firstly pointed out in \cite{dc}):
\begin{equation}\label{propr}
B_{n,k}(trx_1,tr^2x_2,\dots,tr^{n-k+1}x_{n-k+1})=t^kr^nB_{n,k}(x_1,\dots, x_{n-k+1}).
\end{equation}
\begin{equation}\label{prop2}
\begin{split}
B_{n,k+1}(x)=\frac{1}{(k+1)!}\underbrace{\sum_{\alpha_1=k}^{n-1}\sum_{\alpha_2=k-1}^{\alpha_1-1}\cdots \sum_{\alpha_k}^{\alpha_{k-1}-1}}_{k}&\overbrace{{n\choose \alpha_1}{\alpha_1\choose \alpha_2}\cdots{\alpha_{k-1}\choose \alpha_k}}^{k} \cdot\\
&\cdot x_{n-\alpha_1}x_{\alpha_1-\alpha_2}\cdots x_{\alpha_{k-1}-\alpha_k}x_{\alpha_k}.
\end{split}
\end{equation}

The complete Bell polynomials are given by:
$$
Y_n(x_1,\dots, x_n)=\sum_{k=1}^nB_{n,k}(x),\quad Y_0:=0,
$$
and the role they play in our context is given by the following formula \cite[Eq. 3b, p.134]{comtet}:
\begin{equation}\label{exp}
\frac{d^n}{dx^n}\left(\exp\left(\sum_{j=1}^\infty a_j\frac{x^j}{j!}\right)\right)|_0=Y_n(a_1,\dots, a_n).
\end{equation}
Observe that from \eqref{propr} it follows:
\begin{equation}\label{proprY}
Y_n(rx_1,r^2x_2,\dots,r^{n}x_{n})=r^nY_n(x_1,\dots, x_{n}).
\end{equation}
The following lemma has a key role in the proof of Theorem \ref{mainnew}. 
\begin{lem}\label{bjklimit}
Let $a_j=j!/j^2$. Then
$$
\lim_{n\rightarrow\infty}\frac{(2n)^2B_{2n,k+1}(a)}{(2n)!}=\frac{k+1}{(k+1)!}\sum_{j_1=1}^\infty\frac{1}{j_1^2}\cdots\sum_{j_{k-1}=1}^\infty\frac{1}{j_{k}^2}.
$$
\end{lem}
\begin{proof}
Observe first that by \eqref{prop2}, we get:
\begin{equation}\label{bkkkk}
\frac{B_{2n,k+1}(a)}{(2n)!}=\frac{1}{(k+1)!}\underbrace{\sum_{\alpha_1=k}^{2n-1}\sum_{\alpha_2=k-1}^{\alpha_1-1}\cdots \sum_{\alpha_k=1}^{\alpha_{k-1}-1}}_{k}\frac{1}{(2n-\alpha_1)^2(\alpha_1-\alpha_2)^2\cdots (\alpha_{k-1}-\alpha_k)^2\alpha_k^2}.
\end{equation}
We proceed by induction on $k$. For $k=1$ we have:
\begin{equation}
\begin{split}
&\frac{(2n)^2B_{2n,2}(\tilde a)}{(2n)!}=\frac{(2n)^2}{2}\sum_{\alpha_1=1}^{2n-1}\frac{1}{(2n-\alpha_1)^2\alpha_1^2}\\
&=\frac{1}2\left(\frac{(2n)^2}{(2n-1)^2}+\frac{(2n)^2}{(2n-2)^22^2}+\frac{(2n)^2}{(2n-3)^23^2}+\cdots+\frac{(2n)^2}{2^2(2n-2)^2}+\frac{(2n)^2}{(2n-1)^2}\right)\\
&=\frac{(2n)^2}{(2n-1)^2}+\frac{(2n)^2}{(2n-2)^22^2}+\frac{(2n)^2}{(2n-3)^23^2}+\dots+\frac{(2n)^2}{(n-1)^2(n+1)^2}\\
&=\frac{(2n)^2}{(2n-1)^2}\sum_{j=1}^\infty\frac1{j^2}+\psi(n),
\end{split}\nonumber
\end{equation}
where $\psi(n)$ satisfies $\lim_{n\rightarrow \infty}\psi(n)=0$.
Thus we have:
$$
\lim_{n\rightarrow \infty}\frac{(2n)^2B_{2n,2}(\tilde a)}{(2n)!}=\sum_{j=1}^\infty\frac{1}{j^2}.
$$
Assume now that:
\begin{equation}\label{inductive}
\lim_{n\rightarrow \infty}\frac{(2n)^2B_{2n,k}(\tilde a)}{(2n)!}=\frac{k}{k!}\sum_{j_1=1}^\infty\frac{1}{j_1^2}\cdots\sum_{j_{k-1}=1}^\infty\frac{1}{j_{k-1}^2}.
\end{equation}
By \eqref{bkkkk} this implies:
$$
\lim_{n\rightarrow \infty}\underbrace{\sum_{\alpha_1=k-1}^{2n-1}\sum_{\alpha_2=k-2}^{\alpha_1-1}\cdots \sum_{\alpha_{k-1}=1}^{\alpha_{k-2}-1}}_{k-1}\frac{(2n)^2}{(2n-\alpha_1)^2(\alpha_1-\alpha_2)^2\cdots (\alpha_{k-2}-\alpha_{k-1})^2\alpha_{k-1}^2}=k\sum_{j_1=1}^\infty\frac{1}{j_1^2}\cdots\sum_{j_{k-1}=1}^\infty\frac{1}{j_{k-1}^2}
$$
and thus:
\begin{equation}
\begin{split}
\!\!\!\!\!\!\!\!\!\!\!\!\!\!\!\!\!\!\!\!\!\frac{(2n)^2B_{2n,k+1}(\tilde a)}{(2n)!}=&\frac{(2n)^2}{(k+1)!}\underbrace{\sum_{\alpha_1=k}^{2n-1}\sum_{\alpha_2=k-1}^{\alpha_1-1}\cdots \sum_{\alpha_k=1}^{\alpha_{k-1}-1}}_{k}\frac{1}{(2n-\alpha_1)^2(\alpha_1-\alpha_2)^2\cdots (\alpha_{k-1}-\alpha_k)^2\alpha_k^2}\\
=&\frac{(2n)^2}{(k+1)!(2n-k)^2}\underbrace{\sum_{\alpha_2=k-1}^{k-1}\cdots \sum_{\alpha_k=1}^{\alpha_{k-1}-1}}_{k-1}\frac{1}{(k-\alpha_2)^2\cdots (\alpha_{k-1}-\alpha_k)^2\alpha_k^2}+\\
&+\frac{(2n)^2}{(k+1)!(2n-k-1)^2}\underbrace{\sum_{\alpha_2=k-1}^{k}\cdots \sum_{\alpha_k=1}^{\alpha_{k-1}-1}}_{k-1}\frac{1}{(k+1-\alpha_2)^2\cdots (\alpha_{k-1}-\alpha_k)^2\alpha_k^2}+\dots\\
&\dots+\frac{(2n)^2}{(k+1)!n^2}\underbrace{\sum_{\alpha_2=k-1}^{n-1}\cdots \sum_{\alpha_k=1}^{\alpha_{k-1}-1}}_{k-1}\frac{1}{(n-\alpha_2)^2\cdots (\alpha_{k-1}-\alpha_k)^2\alpha_k^2} +\dots\\
&\dots+\frac{1}{(k+1)!2^2}\underbrace{\sum_{\alpha_2=k-1}^{2n-3}\cdots \sum_{\alpha_k=1}^{\alpha_{k-1}-1}}_{k-1}\frac{(2n)^2}{(2n-2-\alpha_2)^2\cdots (\alpha_{k-1}-\alpha_k)^2\alpha_k^2}\\
&\frac{1}{(k+1)!}\underbrace{\sum_{\alpha_2=k-1}^{2n-2}\cdots \sum_{\alpha_k=1}^{\alpha_{k-1}-1}}_{k-1}\frac{(2n)^2}{(2n-1-\alpha_2)^2\cdots (\alpha_{k-1}-\alpha_k)^2\alpha_k^2}\\
=&\frac{(2n)^2}{(k+1)!(2n-k)^2}+\\
&+\frac{(2n)^2}{(k+1)!(2n-k-1)^2}\underbrace{\sum_{\alpha_2=k-1}^{k}\cdots \sum_{\alpha_k=1}^{\alpha_{k-1}-1}}_{k-1}\frac{1}{(k+1-\alpha_2)^2\cdots (\alpha_{k-1}-\alpha_k)^2\alpha_k^2}+\\
&+\frac{(2n)^2}{(k+1)!(2n-k-2)^2}\underbrace{\sum_{\alpha_2=k-1}^{k+1}\cdots \sum_{\alpha_k=1}^{\alpha_{k-1}-1}}_{k-1}\frac{1}{(k+2-\alpha_2)^2\cdots (\alpha_{k-1}-\alpha_k)^2\alpha_k^2}+\\
&+\dots+\Psi(n)+\dots+\frac{k}{(k+1)!}\sum_{j_1=1}^\infty\frac{1}{j_1^2}\cdots\sum_{j_{k-1}=1}^\infty\frac{1}{j_{k}^2}
\end{split}\nonumber
\end{equation}
where $\Psi(n)$ is an infinitesimal function of $n$. 

Then by \eqref{inductive}:
\begin{equation}
\begin{split}
\!\!\!\!\!\!\!\!\!\lim_{n\rightarrow\infty}\frac{(2n)^2B_{2n,k+1}(\tilde a)}{(2n)!}=&\frac{1}{(k+1)!}+\frac{1}{(k+1)!}\underbrace{\sum_{\alpha_2=k-1}^{k}\cdots \sum_{\alpha_k=1}^{\alpha_{k-1}-1}}_{k-1}\frac{1}{(k+1-\alpha_2)^2\cdots (\alpha_{k-1}-\alpha_k)^2\alpha_k^2}+\\
&+\frac{1}{(k+1)!}\underbrace{\sum_{\alpha_2=k-1}^{k+1}\cdots \sum_{\alpha_k=1}^{\alpha_{k-1}-1}}_{k-1}\frac{1}{(k+2-\alpha_2)^2\cdots (\alpha_{k-1}-\alpha_k)^2\alpha_k^2}+\\
&+\dots+\frac{k}{(k+1)!}\sum_{j_1=1}^\infty\frac{1}{j_1^2}\cdots\sum_{j_{k-1}=1}^\infty\frac{1}{j_{k}^2}
\end{split}\nonumber
\end{equation}
i.e.:
\begin{equation}
\begin{split}
\!\!\!\!\!\!\!\!\!\lim_{n\rightarrow\infty}\frac{(2n)^2B_{2n,k+1}(\tilde a)}{(2n)!}=&\frac{1}{(k+1)!}\sum_{j=1}^\infty\underbrace{\sum_{\alpha_2=k-1}^{k-2+j}\cdots \sum_{\alpha_k=1}^{\alpha_{k-1}-1}}_{k-1}\frac{1}{(k-1+j-\alpha_2)^2\cdots (\alpha_{k-1}-\alpha_k)^2\alpha_k^2}+\\
&+\frac{k}{(k+1)!}\sum_{j_1=1}^\infty\frac{1}{j_1^2}\cdots\sum_{j_{k-1}=1}^\infty\frac{1}{j_{k}^2}.
\end{split}\nonumber
\end{equation}
By setting:
$$
A_1(k):=1,\quad A_j(k)=\frac k{j^2}+k\sum_{s=2}^{j-1}\frac{A_{j-s+1}(k-1)}{s^2},
$$
we finally get:
$$
\lim_{n\rightarrow\infty}\frac{(2n)^2B_{2n,k+1}(\tilde a)}{(2n)!}=\frac{1}{(k+1)!}\sum_{j=1}^\infty A_j(k)+\frac{k}{(k+1)!}\sum_{j_1=1}^\infty\frac{1}{j_1^2}\cdots\sum_{j_{k-1}=1}^\infty\frac{1}{j_{k}^2},
$$
and conclusion follows by observing that:
\begin{equation}
\begin{split}
\sum_{j=1}^\infty A_j(k)&=\sum_{j=1}^\infty\frac{k}{j^2}+A_2(k-1)\sum_{j=2}^\infty\frac{k}{j^2}+A_3(k-1)\sum_{j=3}^\infty\frac{k}{j^2}+\dots\\
&=\sum_{j_1=1}^\infty\frac{1}{j_1^2}\cdots\sum_{j_{k-1}=1}^\infty\frac{1}{j_{k}^2}.
\end{split}\nonumber
\end{equation}

\end{proof}

We are now able to prove our main result.
\begin{proof}[Proof of Theorem \ref{mainnew}]

Observe first that if $(M,cg)$ does not admit a K\"ahler immersion into $\CP^\infty$ for any value of $c>0$, then it does not either in any other space form. 
In fact, if $(M,cg)$ admits a K\"ahler immersion into $\ell^2(\mathds{C})$ then by a theorem of Bochner (see \cite{bochner}), it also does into $\CP^\infty$, and in particular since the multiplication by $c$ is harmless when one considers K\"ahler immersion into flat spaces, it does for any value of $c>0$. Further, in a totally similar way as in the proof of Bochner's statement, one can prove that a K\"ahler manifold admitting a K\"ahler immersion into $\CH^\infty$ can also be K\"ahler immersed into $\ell^2(\C)$.
Thus, it is enough to show that $(\mathds{C},cg)$ does not admit a K\"ahler immersion into $\CP^\infty$ for any $c>0$. 

Further, the diastasis function associated to the Cigar metric is globally defined and positive by Lemma \ref{pos} above. 

Then, by Calabi's criterion stated above, it remains only to show that there exists $n$ such that:
$$
\frac{\partial^{2n}\exp\left(cD_0(|z|^2)\right)}{\partial z^n\partial \bar z^n}|_0< 0,
$$
where $D_0(|z|^2)$ is the K\"ahler potential defined in \eqref{expdiast}.
Observe first that setting:
\begin{equation}\label{aj}
\tilde a_j:=-c\frac{j!}{j^2},
\end{equation}
by \eqref{exp} and \eqref{proprY} we get:
\begin{equation}
\begin{split}
\frac{\partial^{2n}\exp\left(cD_0(|z|^2)\right)}{\partial z^n\partial \bar z^n}|_0=&\frac{1}{n!}\frac{d^{n}\exp\left(cD_0(x)\right)}{dx^n}|_0\\
=&\frac{1}{n!}Y_n\left(-\tilde a_1,(-1)^2\tilde a_2,\dots, (-1)^n\tilde a_n\right)\\
=&\frac{(-1)^n}{n!}Y_n\left(\tilde a_1,\dots, \tilde a_n\right).
\end{split}\nonumber
\end{equation}
We wish to prove that for any $c>0$ there exists $n$ big enough such that:
$$
Y_{2n}\left(a_1,\dots, a_{2n}\right)<0.
$$
Observe first that since $\tilde a_j=-c\, a_j$ with $a_j=j!/j^2$, we get:
\begin{equation}
\begin{split}
{Y_{2n}(\tilde a)}=&\sum_{k=1}^{2n}(-1)^{k}c^kB_{2n,k}( a)\\
=&\,\frac{(2n)!c}{(2n)^2}\left(-1+\frac{c(2n)^2B_{2n,2}( a)}{(2n)!}-\frac{c^2(2n)^2B_{2n,3}( a)}{(2n)!}+\dots+\frac{c^{2n-1}(2n)^2}{(2n)!}\right).
\end{split}\nonumber
\end{equation}
Thus, we need to prove that for any value of $c$ there exists $n$ large enough such that the following inequality holds:
\begin{equation}\label{eqtobeproven}
\frac{c(2n)^2B_{2n,2}( a)}{(2n)!}-\frac{c^2(2n)^2B_{2n,3}( a)}{(2n)!}+\dots+\frac{c^{2n-1}(2n)^2}{(2n)!}<1.
\end{equation}
By Lemma \ref{bjklimit},
$$
\lim_{n\rightarrow+\infty}\frac{(2n)^2B_{2n,k+1}( a)}{(2n)!}=\frac{k+1}{(k+1)!}\sum_{j_1=1}^\infty\frac{1}{j_1^2}\sum_{j_2=1}^\infty\frac{1}{j_2^2}\cdots\sum_{j_k=1}^\infty\frac{1}{j_k^2},
$$
and since:
$$
\sum_{j=1}^\infty\frac{1}{j^2}=\frac{\pi^2}{6},
$$
we get:
$$
\lim_{n\rightarrow+\infty}\frac{(2n)^2B_{2n,k+1}( a)}{(2n)!}=\frac{1}{k!}\left(\frac{\pi^2}{6}\right)^k.
$$
Plugging this into \eqref{eqtobeproven}, we get that as $n$ goes to infinity the left hand side converge to:
$$
\sum_{k=1}^\infty\frac{(-1)^{k+1}c^k}{k!}\left(\frac{\pi^2}{6}\right)^k=1-e^{-\frac{c\,\pi^2}{6}},
$$
and conclusion follows by observing that $1-e^{-\frac{c\,\pi^2}{6}}$ is strictly increasing as a function on $c$ and its limit value as $c$ grows is $1$.
\end{proof}

\begin{remark}\rm 
With similar tools one can prove  that  $(\C, cg)$, namely the complex plane with a multiple of the Cigar metric does not admit a K\"ahler immersion into any indefinite flat space of finite signature.
\end{remark}

\end{document}